\documentclass[review]{elsarticle}

\usepackage{mathrsfs,amsmath}
\usepackage{color}
\usepackage{bm}
\usepackage{graphicx}
\usepackage{subfigure}
\usepackage{tikz}
\usepackage{mathtools}
\usetikzlibrary{arrows}
\usepackage{xparse}
\usepackage{amssymb}
\usepackage{graphicx}
\usepackage{amsfonts}
\usepackage{amsthm}
\biboptions{numbers,sort&compress}
\newtheorem{thm}{Theorem}
\newtheorem{defn}{Definition}

\newtheorem{corollary}{Corollary}

\usepackage[justification=centering]{caption}
\RequirePackage[colorlinks,citecolor=blue,urlcolor=blue,linkcolor=blue]{hyperref}
\usepackage{cleveref}

\usepackage{enumerate}

\begin{document}

\begin{frontmatter}



\title{Convolution theorems for the free metaplectic transformation and its application}


\author{Hui Zhao$^{a,b}$}
\author{Bing-Zhao Li$^{a,b}$\corref{mycorrespondingauthor}}
\cortext[mycorrespondingauthor]{Corresponding author}\ead{li\_bingzhao@bit.edu.cn}

\address{$^{a}$School of Mathematics and Statistics, Beijing Institute of Technology, Beijing 100081, China}
\address{$^{b}$Beijing Key Laboratory on MCAACI, Beijing Institute of Technology, Beijing 100081, China}

\begin{abstract}
	
 The free metaplectic transformation (FMT) is widely used in many fields such as filter design, pattern recognition, image processing and optics. In order to obtain a more concise and intuitive convolution form, this paper studies two kinds of new convolution theorems in the FMT domain. First, based on the expression of the generalized translation, we derive the convolution theorem of the first kind in the FMT domain, which has elegance and simplicity comparable to the classical results of Fourier transform (FT). Second, we again give the convolution theorem of the second kind in the FMT domain to further study the diversity of convolution theory. It has the advantage that it can be represented by a simple integral and is easy to implement in multiplying filter designs. Finally, based on the simple form of the above convolution theorems, we discuss multiplicative filters in the MFT domain.
\end{abstract}


\end{frontmatter}


\section{Introduction}
The free metaplectic transformation (FMT) plays an important role in many fields of optics and signal processing \cite{Z. C. Zhang2016Jan,Z. C. Zhang2019,T. K. Garg2021,M. Gosson2006}. FMT was first proposed in \cite{G. B. Folland1989}, and it is a generalization of many widely used signal processing operations, including Fourier transform (FT), fractional Fourier transform (FRFT), linear canonical transform (LCT) and Fresnel transform (FRT). In spirit, FMT is similar to the well-known LCT, therefore, it is also called N-dimensional (N-D) nonseparable LCT. However, unlike conventional LCTs, FMTs are generated by a general $2N\times2N$ free symplectic matrix $\mathbf{M}$ with $N(2N+1)$ degrees of freedom \cite{J. J. Healy2016}. Due to these additional degrees of freedom, FMT is widely used in many fields such as filter design, pattern recognition, image processing, quantum mechanics, optics and electromagnetic wave propagation analysis \cite{R. Jing2020,F. A. Shah2020,Z. C. Zhang2021,F. A. Shah2021,F. A. Shah2022}.

It is well known that convolution theory plays a central role in signals, images and optics, especially in the design and implementation of multiplicative filters \cite{Q. Feng2016,L. Durak2009,X. N. Xu2012,R. J. Marks1993,Q. Xiang2014,W. W. Ma2017,W. C. Sun2022}. In recent years, many generalizations of the classical convolution theory have been made to extend it to a wider field, and a large number of convolution theorems on LCT have been obtained \cite{B. Deng2006,D. Wei2009,D. Wei2012,D. Wei2012 Circuits,J. Shi2014,D. Wei2016,Z. C. Zhang2016}. The above results are all dealing with one-dimensional (1-D) signal problems. Although some scholars are currently studying related theories in multi-dimensional signal space such as LCT convolution and sampling, their research results are still immature. In recent years, findings in the literature have been published aiming to exploit the advantages of LCT to analyze multi-dimensional signals \cite{D. Y. Wei2019,D. Y. Wei2014}. However, most of the existing multi-dimensional LCT studies are in the form of kernel, which are essentially the product of $N$ copies of a usually 1-D LCT kernel. These studies lack the inherent multi-dimensional nature of LCTs. Considering the pure and complex multi-dimensional structure of LCTs, we utilize a pure multi-dimensional kernel derived from a general $2N\times2N$ real symplectic matrix $\mathbf{M}$ to initiate a specialized study in the field of multi-dimensional signals, called for FMT. At present, the theories are still in the development stage in the FMT field, such as convolution, sampling, uncertainty principle, and so on. Therefore, it is of great theoretical significance to study the convolution theorem in the FMT field.

The convolutional theory of FMT is a very interesting topic and more scholars are needed to explore this issue. Obviously, the convolution theorem result of classical FT states that the traditional convolution of two signals is equivalent to the product of signal FTs. \cite{F. A. Shah2021} proposed a new  convolution theorem in the FMT domain based on the definition of N-D FT, but in terms of the form of the convolution theorem, it is more complicated than the classical FT domain convolution theorem, and it is also difficult to apply. F. A. Shah \cite{F. A. Shah2022} generalized the FT classical convolution to the FMT domain, and the derived convolution theorem also contains an additional chip multiplier terms. In order to obtain a more concise and intuitive convolution theory, we use the transform kernel of FMT to propose two kinds of convolution structures, and the results are as elegant and simple as possible compared to the FT. The advantage of this two kinds of convolution structures is that the result is simple, not affected by any chip multiplier terms, and easy to implement.

The contribution of this paper is mainly divided into four parts:
\begin{enumerate}[(1)]
\item 
The innovation of this study is that most of the existing multi-dimensional signal research adopts the form of transform kernel, which is essentially the product of $N$ copies of the usual 1-D transform kernel. However, this paper specifically uses a pure multi-dimensional kernel derived from a real symplectic matrix $\mathbf{M}$ of general $2N\times2N$ to study convolution theory in the field of multi-dimensional signals.
\item
Based on the expression of generalized translation, we derive the convolution theorem of the first kind in the FMT domain. The advantage is that it has elegance and simplicity comparable to the classical results of FT, the disadvantage is that the generalized convolution expression is a triple integral form, and its transformation into a single integral form is very complicated.
\item
Due to the complexity of triple integrals, we again propose the convolution theorem of the second kind in the FMT domain. It has the advantage that it can be represented by a simple integral and is easy to implement in multiplying filter designs. However, it cannot fully maintain the elegance and simplicity comparable to the classical results of FT, with the addition of a contraction factor of $\frac{1}{\sqrt{2}}$ in the FMT domain.
\item  
The proposed two kinds of convolution theorems in this paper are expressed in the form of a simple product of two FMTs in the FMT domain, and are not affected by any chip multiplier terms, although the result of convolution theorem of the second kind has a contraction factor of $\frac{1}{\sqrt{2}}$. Based on the simplicity of this result, we present the design process of the multiplicative filter.
\end{enumerate}

The paper is organized as follows. Section \ref{Preli} provides some basic knowledge of the FMT and convolution theory. In Section \ref{Two}, two kinds of convolution theorems are given in the FMT domain. In Section \ref{Filter}, based on the above two kinds of convolution theorems, we give a multiplicative filter design in the FMT domain. Section \ref{Con} concludes the article.

\section{Preliminaries}
\label{Preli}
In this section, we mainly review some basic facts regarding FMT, generalized translation and general framework of convolution theory, which will be needed throughout the
paper.
 
\subsection{Free metaplectic transformation}  
\label{FMT}
Firstly, we give some necessary background and natation on the FMT .

The FMT of a signal $f(\mathbf{t})$ with a symplectic matrix $\mathbf{M}=\begin{pmatrix} A & B \\ C & D \end{pmatrix}$, where $\det(B)\neq0$,  is defined as \cite{Z. C. Zhang2019,Z. C. Zhang2021}
\begin{align}
	\begin{split}
		\mathcal{L}_{\mathbf{M}}\left[ f(\mathbf{t})\right] (\mathbf{u})=\widehat{f_{\mathbf{M}}}(\mathbf{u})=\int_{\mathbb{R}^{N}}f(\mathbf{t})\mathcal{K}_{\mathbf{M}}(\mathbf{t},\mathbf{u})\rm{d}\mathbf{t},
		\label{3}	
	\end{split}
\end{align}
where the kernel function $	\mathcal{K}_{\mathbf{M}}(\mathbf{t},\mathbf{u})$ is given by
\begin{align}
	\begin{split}
		\mathcal{K}_{\mathbf{M}}(\mathbf{t},\mathbf{u})=\dfrac{1}{|\det(B)|^{\frac{1}{2}}}e^{\pi i\left(\mathbf{u}DB^{-1}\mathbf{u}^{T}+\mathbf{t}B^{-1}A\mathbf{t}^{T} \right)-2\pi i\mathbf{t}B^{-1}\mathbf{u}^{T} },
		\label{4}	
	\end{split}
\end{align}
where $\mathbf{u}=(u_{1},u_{2},\cdots,u_{N})$, $\mathbf{t}=(t_{1},t_{2},\cdots,t_{N})$, and $A=(a_{kl})$, $B=(b_{kl})$, $C=(c_{kl})$, $D=(d_{kl})$ are all $N\times N$ real matrices satisfying the following constraints:
\begin{align}
	\begin{split}
		AB^{T}=BA^{T},\: CD^{T}=DC^{T},\: AD^{T}-BC^{T}=I_{N},\notag	
	\end{split}
\end{align} 
and where $I_{N}$ denotes an $N\times N$ identity matrix. 

The corresponding inverse formula is given by 
\begin{align}
	\begin{split}
	    f(\mathbf{t})=&\mathcal{L}_{\mathbf{M}^{-1}}\left[\widehat{f_{\mathbf{M}}}(\mathbf{u}) \right] (\mathbf{t})\\
	    =&\dfrac{1}{|\det(B)|^{\frac{1}{2}}}\int_{\mathbb{R}^{N}}F_{\mathbf{M}}(\mathbf{u})e^{-\pi i \left( \mathbf{u}B^{-T}D^{T}\mathbf{u}^{T}+\mathbf{t}A^{T}B^{-T}\mathbf{t}^{T}\right) +2\pi i\mathbf{u}B^{-T}\mathbf{t}^{T} } \rm{d}\mathbf{u},
		\label{5}	
	\end{split}
\end{align}
where $\mathbf{M}^{-1}=\begin{pmatrix} D^{T} & -B^{T} \\ -C^{T} & A^{T} \end{pmatrix}$.

The free metaplectic transformation has $N(2N+1)$ degrees of freedom \cite{J. J. Healy2016}. It can be reduced to some well-known integral transformations for specific combinations of blocks found in the symplectic matrix, such as the N-D FT, N-D LCT, N-D FRFT and N-D FRT, which are summarized in Table~\ref{Tab:1}.

\begin{table*}
	\centering
	\caption{Some of the specific cases of the FMT}
	\label{Tab:1}  
	\resizebox{\textwidth}{!}{
	\begin{tabular}{l l}
		\hline\hline\noalign{\smallskip}	
		Free metaplectic matrix $\mathbf{M}=\begin{pmatrix} A & B \\ C & D \end{pmatrix}$ & Free metaplectic transformation  \\
		\noalign{\smallskip}\hline\noalign{\smallskip}
		 $A=D=\mathbf{0}$, $B=-C=I_{N}$  & N-D FT  \\[1ex]
	    $A=diag(a_{11},\cdots,a_{NN})$,
		$B=diag(b_{11},\cdots,b_{NN})$, 
		& N-D separable  LCT \\  $C=diag(c_{11},\cdots,c_{NN})$, $D=diag(d_{11},\cdots,d_{NN})$ \\[1ex]
	    $A=D=diag(\cos\alpha_{1},\cdots,\cos\alpha_{N})$,  & N-D separable FRFT  \\
	    $B=-C=diag(\sin\alpha_{1},\cdots,\sin\alpha_{N})$ \\[1ex]
	    $A=D=I_{N}\cos\alpha$, $B=C=I_{N}\sin\alpha$  & N-D non-separable FRFT  \\[1ex]
	    $A=D=I_{N}$, $B=diag(b_{11},\cdots,b_{NN})$, $C=\mathbf{0}$  & N-D separable FRT  \\ [1ex]
	    $A=D=I_{N}$, $C=\mathbf{0}$  & N-D non-separable FRT  \\
	  \noalign{\smallskip}\hline	
	\end{tabular}}
\end{table*}

\subsection{Generalized translation and general framework of convolution theory}
\label{Generalized}
In the general framework of convolution theory \cite{A. I. Zayed1996,A. I. Zayed1998}, most of the existing convolution theories of integral transforms are not the most concise and intuitive compared with the convolution theory of classical FT, so the general convolution and related theory are proposed \cite{J. Shi2012,D. Wei2009,D. Wei2012,D. Wei2012 Circuits,J. Shi2014,D. Wei2016}. Let's review the related content.
 

If we consider a general signal transform and its Fourier-type inverse give by  
\begin{align}
	\begin{split}
		f(\mathbf{t}) =\int_{\mathbb{R}^{N}}\rho(\mathbf{u})F(\mathbf{u})\mathcal{K}(\mathbf{u},\mathbf{t})\rm{d}\mathbf{u},
		\label{7}	
	\end{split}
\end{align}
\begin{align}
	\begin{split}
		F(\mathbf{u})=\int_{\mathbb{R}^{N}}\rho(\mathbf{t})f(\mathbf{t})\mathcal{K}^{\ast}(\mathbf{u},\mathbf{t})\rm{d}\mathbf{t},
		\label{8}	
	\end{split}
\end{align}
then the $\boldsymbol{\tau}-$generalized translation of the $f(\mathbf{t})$ denoted by $f(\mathbf{t}\theta\boldsymbol{\tau})$ is give by \cite{R. J. Marks1993}
\begin{align}
	\begin{split}
		f(\mathbf{t}\theta\boldsymbol{\tau})=\int_{\mathbb{R}^{N}}\rho(\mathbf{u})F(\mathbf{u})\mathcal{K}(\mathbf{u},\mathbf{t})\mathcal{K}^{\ast}(\mathbf{u},\boldsymbol{\tau})\rm{d}\mathbf{u},
		\label{9}	
	\end{split}
\end{align}
where $\theta$ in the argument of the function $f(\mathbf{t}\theta\boldsymbol{\tau})$ is the generalized delay operator for the generalized translation, $\mathcal{K}(\mathbf{u},\mathbf{t})$ is the kernel of the transform, $\mathcal{K}^{\ast}(\mathbf{u},\boldsymbol{\tau})$ stands for the complex conjugater of $\mathcal{K}(\mathbf{u},\boldsymbol{\tau})$, and $\rho(\mathbf{u})$ is the weight function.

The shift property for N-D FT expressed below
\begin{align}
	\begin{split}
		f(\mathbf{t}-\boldsymbol{\tau})\xleftrightarrow{FT}F(\mathbf{u})e^{i\mathbf{u}\boldsymbol{\tau}^{T}},
		\label{09}	
	\end{split}
\end{align}
is consistent with the above definition of the generalized translation as can be easily verified.

Based on the expression of generalized translation, many generalized convolution theorems based on various transformation kernels have been proposed \cite{J. Shi2012,D. Wei2009,D. Wei2016}. However, the above contents are all discussed in the field of 1-D signals, and the generalized transformation of multi-dimensional signal space is not known yet.

\section{Two kinds of convolution theorems for the FMT }
\label{Two}

In this section, we present two new classes of convolution theorems in the FMT domain. The expressions of the new convolution theorems in the FMT domain are all simple product forms of two FMTs, and are not affected by any chip multiplier terms. They have elegance and resemblance to classical FT results.

\subsection{Convolution theorem of the first kind}
\label{first kind}
For the special case of the FMT, let the weight function $\rho(\mathbf{u})$ taken as unity and the kernel taken equal to the kernel of the MFT given by (\ref{4}), the generalized translation in (\ref{9}) can be expressed as follows:
\begin{align}
	\begin{split}
		f(\mathbf{t}\theta\boldsymbol{\tau})=&\int_{\mathbb{R}^{N}}\rho(\mathbf{u})\widehat{f_{\mathbf{M}}}(\mathbf{u})\mathcal{K}_{\mathbf{M}}(\mathbf{t},\mathbf{u})\mathcal{K}^{\ast}_{\mathbf{M}}(\mathbf{t},\mathbf{u})\rm{d}\mathbf{u}\\
		=& \dfrac{1}{|\det(B)|}\int_{\mathbb{R}^{N}}\widehat{f_{\mathbf{M}}}(\mathbf{u})e^{\pi i\left(\mathbf{u}DB^{-1}\mathbf{u}^{T}+\boldsymbol{\tau}B^{-1}A\boldsymbol{\tau}^{T} \right)-2\pi i\boldsymbol{\tau}B^{-1}\mathbf{u}^{T} }\\
		&\times e^{-\pi i\left(\mathbf{u}DB^{-1}\mathbf{u}^{T}+\mathbf{t}B^{-1}A\mathbf{t}^{T} \right)+2\pi i\mathbf{t}B^{-1}\mathbf{u}^{T} } \rm{d}\mathbf{u}\\
		=&\dfrac{1}{|\det(B)|}\int_{\mathbb{R}^{N}}\widehat{f_{\mathbf{M}}}(\mathbf{u})e^{\pi i\left(\boldsymbol{\tau}B^{-1}A\boldsymbol{\tau}^{T}-\mathbf{t}B^{-1}A\mathbf{t}^{T}\right)+2\pi i \left( \mathbf{t}-\boldsymbol{\tau} \right)B^{-1}\mathbf{u}^{T} }\rm{d}\mathbf{u}.
		\label{10}	
	\end{split}
\end{align}
The function defined by formula (\ref{10}) is called the $\boldsymbol{\tau}-$generalized translation $f(\mathbf{t}\theta\boldsymbol{\tau})$ of the signal $f(\mathbf{t})$ based on FMT.

First, we introduce a new concept of a generalized convolution operator $\star^{\mathbf{M}}$ related to a symplectic matrix $\mathbf{M}$, the proposed new definition is a generalized version of \cite{D. Wei2009}.

\begin{defn}
Let $f(\mathbf{t}\theta\boldsymbol{\tau})$ denotes the $\boldsymbol{\tau}-$generalized translation of $f(\mathbf{t})$ in the FMT domain, and the free metaplectic convolution is denoted by $\star^{\mathbf{M}}$ and is defined as
\begin{align}
	\begin{split}
	\left( f\star^{\mathbf{M}}g\right)(\mathbf{t})=&\int_{\mathbb{R}^{N}}f(\boldsymbol{\tau})g(\mathbf{t}\theta\mathbf{\boldsymbol{\tau}})e^{\pi i\left[\mathbf{u}\left( DB^{-1}-B^{-T}D^{T}\right) \mathbf{u}^{T} \right] }\\
	&\times e^{\pi i\left[\mathbf{t}\left( B^{-1}A-A^{T}B^{-T}\right) \mathbf{t}^{T} \right]-2\pi i\left(  \mathbf{t}B^{-1}\mathbf{u}^{T}-\mathbf{u}B^{-T}\mathbf{t}^{T}\right) }\rm{d}\boldsymbol{\tau}.
	\label{12}	
	\end{split}
\end{align}
\end{defn} 	

Based on the generalized free metaplectic convolution operator $\star^{\mathbf{M}}$, we derive the generalized convolution theorem associated with the FMT as follows.
\begin{thm}
Let $z(\mathbf{t})=\left( f\star^{\mathbf{M}} g \right)(\mathbf{t})$, $Z_{\mathbf{M}}$, $F_{\mathbf{M}}$ and $G_{\mathbf{M}}$ denote the FMT of $z(\mathbf{t})$, $f(\mathbf{t})$ and $g(\mathbf{t})$, respectively. Then, we have 
	\begin{align}
		\begin{split}
			Z_{\mathbf{M}}(\mathbf{u})=F_{\mathbf{M}}(\mathbf{u})\cdot G_{\mathbf{M}}(\mathbf{u}),
			\label{13}	
		\end{split}
	\end{align}	
and 
\begin{align}
	\begin{split}
		z(\mathbf{t})=\mathcal{L}_{\mathbf{M}^{-1}}\left[  F_{{\mathbf{M}}} (\mathbf{u}) G_{\mathbf{M}}(\mathbf{u}) \right]   (\mathbf{t}).
		\label{013}
	\end{split}
\end{align}
Moreover
\begin{align}
	\begin{split}
		\mathcal{L}_{\mathbf{M}}\left[ f(\mathbf{t})\cdot g(\mathbf{t}) \right] (\mathbf{u})=\left(F_{\mathbf{M}}\;\star^{\mathbf{M}}\; G_{\mathbf{M}}\right) (\mathbf{u}).
		\label{18}	
	\end{split}
\end{align}	
\end{thm}

\begin{proof} 
 From (\ref{10}), it follows that	
 \begin{align}
 	\begin{split}
 		g(\mathbf{t}\theta\boldsymbol{\tau})=\dfrac{1}{|\det(B)|}\int_{\mathbb{R}^{N}}G_{\mathbf{M}}(\mathbf{u})e^{\pi i\left(\boldsymbol{\tau}B^{-1}A\boldsymbol{\tau}^{T}-\mathbf{t}B^{-1}A\mathbf{t}^{T}\right)+2\pi i \left( \mathbf{t}-\boldsymbol{\tau} \right)B^{-1}\mathbf{u}^{T} }\rm{d}\mathbf{u}.
 		\label{14}	
 	\end{split}
 \end{align}
Using (\ref{12}) and (\ref{14}), we can derive the following result
\begin{align}
	\begin{split}
		\left( f\star^{\mathbf{M}}g\right)(\mathbf{t})=&\int_{\mathbb{R}^{N}}f(\boldsymbol{\tau})\dfrac{1}{|\det(B)|}\int_{\mathbb{R}^{N}}G_{\mathbf{M}}(\mathbf{u})e^{\pi i\left(\boldsymbol{\tau}B^{-1}A\boldsymbol{\tau}^{T}-\mathbf{t}B^{-1}A\mathbf{t}^{T}\right)}\\
		&\times e^{2\pi i \left( \mathbf{t}-\boldsymbol{\tau} \right)B^{-1}\mathbf{u}^{T}}{\rm{d}\mathbf{u}}\;\times e^{\pi i\left[\mathbf{u}\left( DB^{-1}-B^{-T}D^{T}\right) \mathbf{u}^{T} \right] }\\
		&\times e^{\pi i\left[\mathbf{t}\left( B^{-1}A-A^{T}B^{-T}\right) \mathbf{t}^{T} \right]-2\pi i\left(  \mathbf{t}B^{-1}\mathbf{u}^{T}-\mathbf{u}B^{-T}\mathbf{t}^{T}\right) }\rm{d}\boldsymbol{\tau}.
		\label{15}	
	\end{split}
\end{align}
According to the definition of the FMT, we get
\begin{align}
	\begin{split}
		&Z_{\mathbf{M}}(\mathbf{u})\\
		=&\dfrac{1}{|\det(B)|^{\frac{1}{2}}}\int_{\mathbb{R}^{N}}z(\mathbf{t})e^{\pi i\left(\mathbf{u}DB^{-1}\mathbf{u}^{T}+\mathbf{t}B^{-1}A\mathbf{t}^{T} \right)-2\pi i\mathbf{t}B^{-1}\mathbf{u}^{T} }\rm{d}\mathbf{t}\\
		=&\dfrac{1}{|\det(B)|^{\frac{1}{2}}}\int_{\mathbb{R}^{N}}\int_{\mathbb{R}^{N}}f(\boldsymbol{\tau})\dfrac{1}{|\det(B)|}\int_{\mathbb{R}^{N}}G_{\mathbf{M}}(\mathbf{u})e^{\pi i\left(\boldsymbol{\tau}B^{-1}A\boldsymbol{\tau}^{T}-\mathbf{t}B^{-1}A\mathbf{t}^{T}\right)}\\
		&\times e^{2\pi i \left( \mathbf{t}-\boldsymbol{\tau} \right)B^{-1}\mathbf{u}^{T}} {\rm{d}\mathbf{u}}\times e^{\pi i\left[\mathbf{u}\left( DB^{-1}-B^{-T}D^{T}\right) \mathbf{u}^{T}  +\mathbf{t}\left( B^{-1}A-A^{T}B^{-T}\right) \mathbf{t}^{T}\right]}\\
		&\times e^{-2\pi i\left(  \mathbf{t}B^{-1}\mathbf{u}^{T}-\mathbf{u}B^{-T}\mathbf{t}^{T}\right) }{\rm{d}\boldsymbol{\tau}}\times e^{\pi i\left(\mathbf{u}DB^{-1}\mathbf{u}^{T}+\mathbf{t}B^{-1}A\mathbf{t}^{T} \right)-2\pi i\mathbf{t}B^{-1}\mathbf{u}^{T} }\rm{d}\mathbf{t}.
		\label{117}	
	\end{split}
	\end{align}
From the FMT’s inverse formula, that is
\begin{align}
	\begin{split}
		f(\mathbf{t})=\dfrac{1}{|\det(B)|^{\frac{1}{2}}}\int_{\mathbb{R}^{N}}F_{\mathbf{M}}(\mathbf{u})e^{-\pi i \left( \mathbf{u}B^{-T}D^{T}\mathbf{u}^{T}+\mathbf{t}A^{T}B^{-T}\mathbf{t}^{T}\right) +2\pi i\mathbf{u}B^{-T}\mathbf{t}^{T} } \rm{d}\mathbf{u}.
		\label{126}	
	\end{split}
\end{align}
Hence, by (\ref{117}) and (\ref{126}), we have
\begin{align}
	\begin{split}
		Z_{\mathbf{M}}(\mathbf{u})=&\left\lbrace \dfrac{1}{|\det(B)|^{\frac{1}{2}}}\int_{\mathbb{R}^{N}}f(\boldsymbol{\tau})e^{\pi i\left(\mathbf{u}DB^{-1}\mathbf{u}^{T}+\boldsymbol{\tau}B^{-1}A\boldsymbol{\tau}^{T} \right)-2\pi i\boldsymbol{\tau}B^{-1}\mathbf{u}^{T} }\rm{d}\boldsymbol{\tau}\right\rbrace \\
		&\times \left\lbrace \dfrac{1}{|\det(B)|^{\frac{1}{2}}}\int_{\mathbb{R}^{N}}G_{\mathbf{M}}(\mathbf{u})e^{-\pi i \left( \mathbf{u}B^{-T}D^{T}\mathbf{u}^{T}+\mathbf{t}A^{T}B^{-T}\mathbf{t}^{T}\right) +2\pi i\mathbf{u}B^{-T}\mathbf{t}^{T} } \rm{d}\mathbf{u}\right\rbrace \\
		&\times \dfrac{1}{|\det(B)|^{\frac{1}{2}}}\int_{\mathbb{R}^{N}}e^{\pi i\left(\mathbf{u}DB^{-1}\mathbf{u}^{T}+\mathbf{t}B^{-1}A\mathbf{t}^{T} \right)-2\pi i\mathbf{t}B^{-1}\mathbf{u}^{T} }\rm{d}\mathbf{t}\\
		&=F_{\mathbf{M}}(\mathbf{u})\left\lbrace \dfrac{1}{|\det(B)|^{\frac{1}{2}}}\int_{\mathbb{R}^{N}}g(\mathbf{t})e^{\pi i\left(\mathbf{u}DB^{-1}\mathbf{u}^{T}+\mathbf{t}B^{-1}A\mathbf{t}^{T} \right)-2\pi i\mathbf{t}B^{-1}\mathbf{u}^{T} }\rm{d}\mathbf{t}\right\rbrace \\
		&=F_{\mathbf{M}}(\mathbf{u})\cdot G_{\mathbf{M}}(\mathbf{u}).	
		\label{17}	
	\end{split}
\end{align}	
In view of (\ref{17}) and the reversible property of FMT, we get
\begin{align}
	\begin{split}
		z(\mathbf{t})=\mathcal{L}_{\mathbf{M}^{-1}}\left[   F_{{\mathbf{M}}} (\mathbf{u}) G_{\mathbf{M}}(\mathbf{u}) \right]   (\mathbf{t}).\notag
	\end{split}
\end{align}	
The proof of (\ref{13}) and (\ref{013}) is completed. The proof of (\ref{18}) is given below. \\
According to (\ref{8}), we obtain	
\begin{align}
	\begin{split}
		\widehat{f_{\mathbf{M}}}(\mathbf{u}\theta\boldsymbol{\tau})=&\int_{\mathbb{R}^{N}}\rho(\mathbf{t})f(\mathbf{t})\mathcal{K}_{\mathbf{M}^{-1}}(\mathbf{t},\mathbf{u})\mathcal{K}^{\ast}_{\mathbf{M}^{-1}}(\mathbf{t},\boldsymbol{\tau})\rm{d}\mathbf{t}\\
		=& \dfrac{1}{|\det(B)|}\int_{\mathbb{R}^{N}}f(\mathbf{t})e^{-\pi i\left(\boldsymbol{\tau}B^{-T}D^{T}\boldsymbol{\tau}^{T}+\mathbf{t}A^{T}B^{-T}\mathbf{t}^{T} \right)+2\pi i\boldsymbol{\tau}B^{-T}\mathbf{t}^{T} }\\
		&\times e^{\pi i\left(\mathbf{u}B^{-T}D^{T}\mathbf{u}^{T}+\mathbf{t}A^{T}B^{-T}\mathbf{t}^{T} \right)-2\pi i\mathbf{u}B^{-T}\mathbf{t}^{T} } \rm{d}\mathbf{t}\\
		=&\dfrac{1}{|\det(B)|}\int_{\mathbb{R}^{N}}f(\mathbf{t})e^{\pi i\left(\mathbf{u}B^{-T}D^{T}\mathbf{u}^{T}-\boldsymbol{\tau}B^{-T}D^{T}\boldsymbol{\tau}^{T}\right) +2\pi i \left( \boldsymbol{\tau}-\mathbf{u} \right) B^{-T}\mathbf{t}^{T} }\rm{d}\mathbf{t}.
		\label{010}	
	\end{split}
\end{align}
From (\ref{010}), it follows that
\begin{align}
	\begin{split}
		G_{\mathbf{M}}(\mathbf{u}\theta\boldsymbol{\tau})=&\,\widehat{g_{\mathbf{M}}}(\mathbf{u}\theta\boldsymbol{\tau})\\
		=&\dfrac{1}{|\det(B)|}\int_{\mathbb{R}^{N}}g(\mathbf{t})e^{\pi i\left(\mathbf{u}B^{-T}D^{T}\mathbf{u}^{T}-\boldsymbol{\tau}B^{-T}D^{T}\boldsymbol{\tau}^{T}\right) +2\pi i \left( \boldsymbol{\tau}-\mathbf{u} \right) B^{-T}\mathbf{t}^{T} }\rm{d}\mathbf{t}.
		\label{19}	
	\end{split}
\end{align}
Using (\ref{12}) and (\ref{19}), we can derive the following result	
\begin{align}
	\begin{split}
		&\left(F_{\mathbf{M}}\;\star^{\mathbf{M}}\; G_{\mathbf{M}}\right) (\mathbf{u})\\
		=&\int_{\mathbb{R}^{N}}F_{\mathbf{M}}(\boldsymbol{\tau})G_{\mathbf{M}}(\mathbf{u}\theta\mathbf{\boldsymbol{\tau}})e^{\pi i\left[\mathbf{u}\left( DB^{-1}-B^{-T}D^{T}\right) \mathbf{u}^{T}+\mathbf{t}\left( B^{-1}A-A^{T}B^{-T}\right) \mathbf{t}^{T} \right] }\\
		&\times e^{-2\pi i\left(  \mathbf{t}B^{-1}\mathbf{u}^{T}-\mathbf{u}B^{-T}\mathbf{t}^{T}\right) }\rm{d}\boldsymbol{\tau}\\
		=&\int_{\mathbb{R}^{N}}F_{\mathbf{M}}(\boldsymbol{\tau})\dfrac{1}{|\det(B)|}\int_{\mathbb{R}^{N}}g(\mathbf{t})e^{\pi i\left(\mathbf{u}B^{-T}D^{T}\mathbf{u}^{T}-\boldsymbol{\tau}B^{-T}D^{T}\boldsymbol{\tau}^{T}\right) }\\
		&\times e^{2\pi i \left( \boldsymbol{\tau}-\mathbf{u} \right) B^{-T}\mathbf{t}^{T} }{\rm{d}\mathbf{t}}\times e^{\pi i\left[\mathbf{u}\left( DB^{-1}-B^{-T}D^{T}\right) \mathbf{u}^{T}+\mathbf{t}\left( B^{-1}A-A^{T}B^{-T}\right) \mathbf{t}^{T} \right] }\\
		&\times e^{-2\pi i\left(  \mathbf{t}B^{-1}\mathbf{u}^{T}-\mathbf{u}B^{-T}\mathbf{t}^{T}\right) }\rm{d}\boldsymbol{\tau}\\
		=&\left\lbrace \dfrac{1}{|\det(B)|^{\frac{1}{2}}}\int_{\mathbb{R}^{N}}F_{\mathbf{M}}(\boldsymbol{\tau})e^{-\pi i \left( \boldsymbol{\tau}B^{-T}D^{T}\boldsymbol{\tau}^{T}+\mathbf{t}A^{T}B^{-T}\mathbf{t}^{T}\right) +2\pi i\boldsymbol{\tau}B^{-T}\mathbf{t}^{T} } \rm{d}\boldsymbol{\tau}\right\rbrace \\
		&\times \dfrac{1}{|\det(B)|^{\frac{1}{2}}}\int_{\mathbb{R}^{N}}g(\mathbf{t})e^{\pi i\left(\mathbf{u}DB^{-1}\mathbf{u}^{T}+\mathbf{t}B^{-1}A\mathbf{t}^{T} \right)-2\pi i\mathbf{t}B^{-1}\mathbf{u}^{T} }\rm{d}\mathbf{t}\\
		=&\dfrac{1}{|\det(B)|^{\frac{1}{2}}}\int_{\mathbb{R}^{N}}f(\mathbf{t})g(\mathbf{t})e^{\pi i\left(\mathbf{u}DB^{-1}\mathbf{u}^{T}+\mathbf{t}B^{-1}A\mathbf{t}^{T} \right)-2\pi i\mathbf{t}B^{-1}\mathbf{u}^{T} }\rm{d}\mathbf{t} \\
		=&\mathcal{L}_{\mathbf{M}}\left[ f(\mathbf{t})\cdot g(\mathbf{t}) \right] (\mathbf{u}).
		\label{20}	
	\end{split}
\end{align}		
\end{proof}	

For the matrix $\mathbf{M}=(A,B;C,D)$, we can obtain different convolution theories by choosing appropriate matrix parameters, as well as the convolution theorems corresponding to all integral transformations shown in Table~\ref{Tab:1}. The special cases of Theorem 1 in the FT, FRFT and LCT domains are given below.

\begin{corollary}
	For $A=D=\mathbf{0}$ and $B=-C=I_{N}$, the (\ref{13}) reduces to the ordinary convolution theorem for the classical N-D FT $F(\mathbf{u})$,	
	\begin{align}
		\begin{split}
			\int_{\mathbb{R}^{N}}f(\boldsymbol{\tau})g(\mathbf{t}-\boldsymbol{\tau}){\rm{d}\boldsymbol{\tau}}\xleftrightarrow{FT}F(\mathbf{u})\:G(\mathbf{u})
		\label{21}
		\end{split}
	\end{align}
\end{corollary}

\begin{corollary}
	For $A=D=diag(\cos\alpha_{1},\cdots,\cos\alpha_{N})$ and $B=-C=diag(\sin\alpha_{1},\\ \cdots,\sin\alpha_{N})$, the (\ref{13}) reduces to the convolution theorem for the N-D FRFT $F_{\alpha}(\mathbf{u})$,	
	\begin{align}
		\begin{split}
			\int_{\mathbb{R}^{N}}f(\boldsymbol{\tau})g(\mathbf{t}\theta\mathbf{\boldsymbol{\tau}}){\rm{d}\boldsymbol{\tau}}\xleftrightarrow{FRFT}F_{\alpha}(\mathbf{u})\:G_{\alpha}(\mathbf{u})
		\label{22}
		\end{split}
	\end{align}
\end{corollary}	
	
\begin{corollary}
	For  $A=diag(a_{11},\cdots,a_{NN})$, $B=diag(b_{11},\cdots,b_{NN})$, $C=diag(c_{11},\cdots,c_{NN})$ and $D=diag(d_{11},\cdots,d_{NN})$, the (\ref{13}) reduces to the generalized convolution theorem for the N-D LCT $F_{(A,B,C,D)}(\mathbf{u})$,	
	\begin{align}
		\begin{split}
			\int_{\mathbb{R}^{N}}f(\boldsymbol{\tau})g(\mathbf{t}\theta\mathbf{\boldsymbol{\tau}}){\rm{d}\boldsymbol{\tau}}\xleftrightarrow{LCT}F_{(A,B,C,D)}(\mathbf{u})\:G_{(A,B,C,D)}(\mathbf{u})
		\label{23}
		\end{split}
	\end{align}
\end{corollary}	

Theorem 1 states that the generalized convolution of two signals in the spatial domain is equivalent to the simple multiplication of their FMTs in the FMT domain. The advantage of this convolution theorem is that it has elegance and simplicity comparable to the classical results of FT, and is easy to implement, especially in filter design. The disadvantage is that the generalized convolution expression proposed in the article is a triple integral form, and it is very complicated to convert it into a single integral form.

\subsection{Convolution theorem of the second kind}
\label{second kind}
Since the generalized convolution in Subsection \ref{first kind} is a triple integral, it is complicated to reduce the expression to a single integral form. Therefore, we again introduce a new convolution structure in this subsection. The advantage of the new convolution structure is that it can be represented by a simple integral in the FMF domain and is easy to implement in multiplicative filter designs.

Below, we can also give a convolution operator $\ast^{\mathbf{M}}$ associated with the FMT.
\begin{defn}
For any two signals $f(\mathbf{t})$ and $g(\mathbf{t})$, let us define the free metaplectic convolution $\ast^{\mathbf{M}}$ by
\begin{align}
	\begin{split}
		\left( f\ast^{\mathbf{M}}g\right)(\mathbf{t})=\dfrac{\sqrt{2}}{|\det(B)|^{\frac{1}{2}}}\int_{\mathbb{R}^{N}}f(\boldsymbol{\tau})g(\sqrt{2}\mathbf{t}-\mathbf{\boldsymbol{\tau}})e^{2\pi i \left(\frac{\mathbf{t}}{\sqrt{2}}-\boldsymbol{\tau} \right)B^{-1}A\left(\frac{\mathbf{t}}{\sqrt{2}}-\boldsymbol{\tau} \right)^{T} }\rm{d}\boldsymbol{\tau}.
		\label{24}	
	\end{split}
\end{align}
\end{defn} 	

A kind of new convolution theorem for FMT is deduced as shown in the following theorem.

\begin{thm}
Let $z(\mathbf{t})=\left( f\ast^{\mathbf{M}} g \right)(\mathbf{t})$, $Z_{\mathbf{M}}$, $F_{\mathbf{M}}$ and $G_{\mathbf{M}}$ denote the FMT of $z(\mathbf{t})$, $f(\mathbf{t})$ and $g(\mathbf{t})$, respectively. Hence
	\begin{align}
		\begin{split}
			Z_{\mathbf{M}}(\mathbf{u})=F_{\mathbf{M}}\left( \frac{\mathbf{u}}{\sqrt{2}}\right) \cdot G_{\mathbf{M}}\left( \dfrac{\mathbf{u}}{\sqrt{2}}\right),
			\label{25}	
		\end{split}
	\end{align}	
	and 
	\begin{align}
		\begin{split}
			z(\mathbf{t})=\mathcal{L}_{\mathbf{M}^{-1}}\left[   F_{\mathbf{M}}\left( \frac{\mathbf{u}}{\sqrt{2}}\right) \cdot G_{\mathbf{M}}\left( \dfrac{\mathbf{u}}{\sqrt{2}}\right) \right](\mathbf{t}).
			\label{26}
		\end{split}
	\end{align}
\end{thm}

\begin{proof} 
According to the definition of the FMT and (\ref{24}), we obtain
\begin{align}
	\begin{split}
		&Z_{\mathbf{M}}(\mathbf{u})\\
		=&\dfrac{1}{|\det(B)|^{\frac{1}{2}}}\int_{\mathbb{R}^{N}}z(\mathbf{t})e^{\pi i\left(\mathbf{u}DB^{-1}\mathbf{u}^{T}+\mathbf{t}B^{-1}A\mathbf{t}^{T} \right)-2\pi i\mathbf{t}B^{-1}\mathbf{u}^{T} }\rm{d}\mathbf{t}\\
		=&\dfrac{1}{|\det(B)|^{\frac{1}{2}}}\int_{\mathbb{R}^{N}}\left\lbrace \dfrac{\sqrt{2}}{|\det(B)|^{\frac{1}{2}}}\int_{\mathbb{R}^{N}}f(\boldsymbol{\tau})g(\sqrt{2}\mathbf{t}-\mathbf{\boldsymbol{\tau}})e^{2\pi i \left(\frac{\mathbf{t}}{\sqrt{2}}-\boldsymbol{\tau} \right)B^{-1}A\left(\frac{\mathbf{t}}{\sqrt{2}}-\boldsymbol{\tau} \right)^{T} }\rm{d}\boldsymbol{\tau} \right\rbrace\\ 
		&\times e^{\pi i\left(\mathbf{u}DB^{-1}\mathbf{u}^{T}+\mathbf{t}B^{-1}A\mathbf{t}^{T} \right)-2\pi i\mathbf{t}B^{-1}\mathbf{u}^{T} }\rm{d}\mathbf{t}\\
		=&\dfrac{\sqrt{2}}{|\det(B)|}\int_{\mathbb{R}^{N}}\int_{\mathbb{R}^{N}}f(\boldsymbol{\tau})g(\sqrt{2}\mathbf{t}-\mathbf{\boldsymbol{\tau}}) e^{\pi i\left(\mathbf{u}DB^{-1}\mathbf{u}^{T}+\mathbf{t}B^{-1}A\mathbf{t}^{T} \right)-2\pi i\mathbf{t}B^{-1}\mathbf{u}^{T} }\\
		&\times e^{2\pi i \left[  \frac{\mathbf{t}}{\sqrt{2}} B^{-1}A(\frac{\mathbf{t}}{\sqrt{2}})^{T}-\frac{\mathbf{t}}{\sqrt{2}} B^{-1}A\boldsymbol{\tau}^{T}-\boldsymbol{\tau} B^{-1}A(\frac{\mathbf{t}}{\sqrt{2}})^{T}+\boldsymbol{\tau} B^{-1}A\boldsymbol{\tau}^{T} \right] }\rm{d}\boldsymbol{\tau}\rm{d}\mathbf{t}.
		\label{27}	
	\end{split}
\end{align}
Let $\lambda=\sqrt{2}\mathbf{t}$, the (\ref{27}) becomes
\begin{align}
	\begin{split}
		&Z_{\mathbf{M}}(\mathbf{u})\\
		=&\dfrac{1}{|\det(B)|}\int_{\mathbb{R}^{N}}\int_{\mathbb{R}^{N}}f(\boldsymbol{\tau})g(\lambda-\mathbf{\boldsymbol{\tau}})e^{\pi i\left[ \mathbf{u}DB^{-1}\mathbf{u}^{T}+\frac{\lambda}{\sqrt{2}}B^{-1}A\left(\frac{\lambda}{\sqrt{2}} \right)^{T} \right] -2\pi i\frac{\lambda}{\sqrt{2}}B^{-1}\mathbf{u}^{T} }\\
		&\times e^{2\pi i \left[  \frac{\lambda}{2} B^{-1}A(\frac{\lambda}{2})^{T}-\frac{\lambda}{2} B^{-1}A\boldsymbol{\tau}^{T}-\boldsymbol{\tau} B^{-1}A(\frac{\lambda}{2})^{T}+\boldsymbol{\tau} B^{-1}A\boldsymbol{\tau}^{T} \right] }\rm{d}\boldsymbol{\tau}\rm{d}\lambda\\
		=&\dfrac{1}{|\det(B)|^{\frac{1}{2}}}\int_{\mathbb{R}^{N}}g(\lambda-\mathbf{\boldsymbol{\tau}})e^{\pi i\frac{\mathbf{u}}{\sqrt{2}}DB^{-1}(\frac{\mathbf{u}}{\sqrt{2}})^{T}+\pi i \left(\lambda-\boldsymbol{\tau} \right)B^{-1}A\left(\lambda-\boldsymbol{\tau} \right)^{T}-2\pi i\left(\lambda-\boldsymbol{\tau} \right)B^{-1}\left(\frac{\mathbf{u}}{\sqrt{2}} \right)^{T}}\\
		&\times {\rm{d}\lambda}\:\dfrac{1}{|\det(B)|^{\frac{1}{2}}}\int_{\mathbb{R}^{N}}f(\boldsymbol{\tau})e^{\frac{\pi i}{2}\mathbf{u}DB^{-1}\mathbf{u}^{T}+\pi i\boldsymbol{\tau}B^{-1}A\boldsymbol{\tau}^{T}-2\pi i\boldsymbol{\tau}B^{-1}\left( \frac{\mathbf{u}}{\sqrt{2}}\right)^{T}}\rm{d}\boldsymbol{\tau}\\
		=&G_{\mathbf{M}}\left( \dfrac{\mathbf{u}}{\sqrt{2}}\right)\dfrac{1}{|\det(B)|^{\frac{1}{2}}}\int_{\mathbb{R}^{N}}f(\boldsymbol{\tau})e^{\pi i\frac{\mathbf{u}}{\sqrt{2}}DB^{-1}\left( \frac{\mathbf{u}}{\sqrt{2}}\right) ^{T}+\pi i\boldsymbol{\tau}B^{-1}A\boldsymbol{\tau}^{T}-2\pi i\boldsymbol{\tau}B^{-1}\left( \frac{\mathbf{u}}{\sqrt{2}}\right)^{T}}\rm{d}\boldsymbol{\tau}\\
		=&F_{\mathbf{M}}\left( \dfrac{\mathbf{u}}{\sqrt{2}}\right)G_{\mathbf{M}}\left( \dfrac{\mathbf{u}}{\sqrt{2}}\right).\nonumber 	
	\end{split}
\end{align}
\end{proof} 

Similarly, some of well-known results about the convolution theorem of the second kind in the FT, FRFT and LCT domains are shown to be special cases of our achieved results.

\begin{corollary}
	For $A=D=\mathbf{0}$ and $B=-C=I_{N}$, the (\ref{25}) reduces to the ordinary convolution theorem for the classical N-D FT $F(\mathbf{u})$,	
	\begin{align}
		\begin{split}
			\sqrt{2}\int_{\mathbb{R}^{N}}f(\boldsymbol{\tau})g(\sqrt{2}\mathbf{t}-\boldsymbol{\tau}){\rm{d}\boldsymbol{\tau}}\xleftrightarrow{FT}F(\dfrac{\mathbf{u}}{\sqrt{2}})\:G\left( \dfrac{\mathbf{u}}{\sqrt{2}}\right) 
			\label{21}
		\end{split}
	\end{align}
\end{corollary}

\begin{corollary}
	For $A=D=diag(\cos\alpha_{1},\cdots,\cos\alpha_{N})$ and $B=-C=diag(\sin\alpha_{1},\\ \cdots,\sin\alpha_{N})$, the (\ref{25}) reduces to the convolution theorem for the N-D FRFT $F_{\alpha}(\mathbf{u})$,	
	\begin{align}
		\begin{split}
			\left( f\ast^{\mathbf{M}} g \right)(\mathbf{t})\xleftrightarrow{FRFT}F_{\alpha}\left( \dfrac{\mathbf{u}}{\sqrt{2}}\right)\:G_{\alpha}\left( \dfrac{\mathbf{u}}{\sqrt{2}}\right)
			\label{22}
		\end{split}
	\end{align}
\end{corollary}	

\begin{corollary}
	For  $A=diag(a_{11},\cdots,a_{NN})$, $B=diag(b_{11},\cdots,b_{NN})$, $C=diag(c_{11},\cdots,c_{NN})$ and $D=diag(d_{11},\cdots,d_{NN})$, the (\ref{25}) reduces to the generalized convolution theorem for the N-D LCT $F_{(A,B,C,D)}(\mathbf{u})$,	
	\begin{align}
		\begin{split}
			\left( f\ast^{\mathbf{M}} g \right)(\mathbf{t})\xleftrightarrow{LCT}F_{(A,B,C,D)}\left( \dfrac{\mathbf{u}}{\sqrt{2}}\right)\:G_{(A,B,C,D)}\left( \dfrac{\mathbf{u}}{\sqrt{2}}\right)
			\label{23}
		\end{split}
	\end{align}
\end{corollary}	

The advantage of Theorem 2 is that it can be represented by a simple integral and is easy to implement in multiplying filter designs. But it does not fully maintain the elegance and simplicity comparable to the classical results of FT, with the addition of a contraction factor of $\frac{1}{\sqrt{2}}$ in the FMT domain.

The two kinds of convolution theorems in this paper are expressed in the FMT domain as simple product forms of two FMTs, which are not affected by any chirp multiplier terms. Based on this result maintaining the elegance and simplicity comparable to the classical results of FT, we present the design process of the multiplicative filter in the next section.

\section{Filter design in the FMT domain }
\label{Filter}
The filter design theory in the field of multi-dimensional signal analysis is still in its infancy, and it is one of the important research contents in the field of digital signal processing. This section mainly discusses multiplicative filters in the FMT domain, and shows that the new convolution structures are easy to implement in the designing of filters.
\begin{figure}[htbp]
	\centering
	\includegraphics{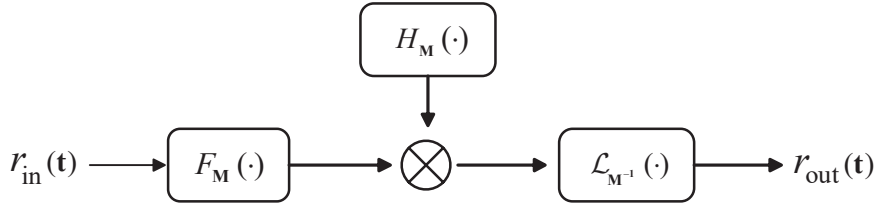}
	\caption{The multiplicative filter in the FMT domain}
\end{figure} 

The multiplicative filter model in the FMT domain is shown in Figure 1. If $F_{\mathbf{M}}\left( \cdot \right) $ is defined as the FMT spectrum of the received signal $r_{\mathrm{in}}(\mathbf{t})$ and $H_{\mathbf{M}}(\cdot)$ is defined as the frequency response of the filter, then the multiplicative effect in Figure 1 can be expressed as
\begin{align}
	\begin{split}
		r_{\mathrm{out}}(\mathbf{t})=\mathcal{L}_{\mathbf{M}^{-1}}\Big(  F_{{\mathbf{M}}}\left[ r_{\mathrm{in}}(\mathbf{t}) \right] (\widetilde{\mathbf{u}})\cdot H_{\mathbf{M}}(\widetilde{\mathbf{u}}) \Big)  (\mathbf{t}),
	\end{split}
\end{align}
where $\widetilde{\mathbf{u}}$ stands for a generalized variable.  From the above analysis, we find that the multiplicative filtering process in the FMT domain basically relies on two steps: first, choose a suitable filter impulse response, multiply it by the MFT spectrum of the input signal $r_{\mathrm{in}}(\mathbf{t})$; second, apply the inverse FMT to the synthesized spectrum for obtaining the desired output signal $r_{\mathrm{out}}(\mathbf{t})$. 

As can be seen from Figure 1, by designing $H_{\mathbf{M}}(\widetilde{\mathbf{u}})$, we can realize many models of multiplicative filters, such as low pass, high pass, band pass, band stop, etc. In this section, we will give an example to explain this. Suppose that the received signal $r_{\mathrm{in}}(\mathbf{t})$ is composed of the desired signal $f(\mathbf{t})$ and the noise $\varepsilon(\mathbf{t})$, namely
\begin{align}
	\begin{split}
		r_{\mathrm{in}}(\mathbf{t})=f(\mathbf{t})+\varepsilon(\mathbf{t}).
	\end{split}
\end{align}
Let $\mathbf{\Omega}_{1}$ and $\mathbf{\Omega}_{2}$ be our region of interest in the MFT domain satisfying $\mathbf{\Omega}_{1}\subseteq\mathbf{\Omega}_{2}$. Assuming a reasonable assumption that the FMT spectra of the signals $r_{\mathrm{in}}(\mathbf{t})$, $f(\mathbf{t})$ and $\varepsilon(\mathbf{t})$ have either no or minimal overlapping. We can use the two kinds of convolution theorems to design a multiplicative filter in the FMT domain to improve the Signal-to-Noise Ratio (SNR), retain the desired signal, and largely discard the noise. 

The \cite{F. A. Shah2021,F. A. Shah2022} give two different convolution theorems for FMT, however, their proposed convolution theorems are not as simple as our result. Hence, in Section \ref{Two}, we show that the convolution of two signals in the spatial domain is equivalent to their simple multiplication in the FMT domain. Our method maintains the same elegance and simplicity as FT and is easier to implement, especially in filter design. According to the two kinds of convolution theorems in Section \ref{Two}, we give the selection method of the multiplicative filter.
\begin{enumerate}[(1)]
	\item 
	By the convolution theorem of the first kind (\ref{13}), it is assumed that the function $G_{\mathbf{M}}(\mathbf{u})$ in (\ref{13}) is the transfer function $H_{\mathbf{M}}(\mathbf{u})$ of the multiplicative filter, the function $H_{\mathbf{M}}(\mathbf{u})$ can be selected as
	\begin{align}
		\begin{split}
			G_{\mathbf{M}}(\mathbf{u})=H_{\mathbf{M}}(\mathbf{u}),
		\end{split}
	\end{align}
	so that it is constant over the domain $\mathbf{\Omega}_{1}$ and is zero or decays rapidly outside $\mathbf{\Omega}_{1}$. 
	\item 
	
	 According to the convolution theorem of the second kind (\ref{25}), with the assumption that function $G_{\mathbf{M}}(\frac{\mathbf{u}}{\sqrt{2}})$ in (\ref{25}), acts as the transfer function of the multiplicative filter, $H_{\mathbf{M}}(\frac{\mathbf{u}}{\sqrt{2}})$ is designed as
	\begin{align}
		\begin{split}
			G_{\mathbf{M}}\left( \frac{\mathbf{u}}{\sqrt{2}}\right) =H_{\mathbf{M}}\left( \frac{\mathbf{u}}{\sqrt{2}}\right) ,
		\end{split}
	\end{align}
	so that it is constant over the domain $\mathbf{\Omega}_{2}$ and is zero or decays rapidly outside $\mathbf{\Omega}_{2}$. 
\end{enumerate}
It is evident from Figure 1 that the output produces only part of the $f(\mathbf{t})$ spectrum, which lies approximately in region $\mathbf{\Omega}_{1}$ or $\mathbf{\Omega}_{2}$. Here, any complexity of domain  $\mathbf{\Omega}_{1}$ or $\mathbf{\Omega}_{2}$ is not considered. Finally, we can exploit the invertibility of the FMT to recover the desired signal $f(\mathbf{t})$. This is obviously easier to implement than the multiplicative filter designed by the convolution theorems proposed in \cite{F. A. Shah2021,F. A. Shah2022}.

\section{Conclusions}
\label{Con}
In this paper, two kinds of new convolution theorems in the FMT domain are studied in detail, and their applications are given. First, based on the generalized translation formulation, we give convolution theorem of the first kind in the FMT domain. It has elegance and simplicity comparable to the classical results of FT, however proposed convolution theorem in the paper is a triple integral form and is computationally complex. Second, to further investigate the diversity of convolution theories, we again give convolution theorem of the second kind in the FMT domain. It has the advantage that it can be represented by a simple 1-D integral and is easy to implement in multiplying filter designs. The two kinds of convolution theorems proposed in this paper are expressed in the FMT domain as simple product forms of two FMTs, and are not affected by any chip product terms. Finally, we discuss multiplicative filters in the FMT domain. In future work, we will consider the sampling theorem of MFT in optics and signal processing.\

\bibliographystyle{elsarticle-num}

\end{document}